\documentclass[12pt]{amsart}
\usepackage{amssymb,latexsym,amsmath,amscd,amsthm,graphicx, hyperref, comment,enumerate}
\usepackage[all]{xy}
\theoremstyle{plain}

\newtheorem{theorem}[subsection]{Theorem}

\newtheorem{lemma}[subsection]{Lemma}

\newtheorem{claim}[subsection]{Claim}
\theoremstyle{definition}
\newtheorem{prop}[subsection]{Proposition}

\newtheorem{remark}[subsection]{Remark}

\newtheorem{note}[subsection]{Note}
\newtheorem{example}[subsection]{Example}

\newcommand{\uu}{\cup}
\newcommand{\ii}{\cap}
\newcommand{\UU}{\bigcup}
\newcommand{\II}{\bigcap}

\newcommand{\sci}{\subset}
\newcommand{\es}{\emptyset}

\newcommand{\ga}{\alpha}
\newcommand{\gb}{\beta}

\renewcommand{\gg}{\gamma}
\newcommand{\gh}{\eta}

\newcommand{\gk}{\kappa}

\newcommand{\gn}{\nu}
\newcommand{\go}{\omega}

\newcommand{\gq}{\theta}
\newcommand{\gs}{\sigma}
\newcommand{\gt}{\tau}

\newcommand{\gG}{\Gamma}

\newcommand{\tit}{\textit}

\newcommand{\C}[1]{\mathcal{#1}}
\newcommand{\D}[1]{\mathbb{#1}}

\newcommand{\te}{\text}

\newcommand{\ol}{\overline}
\newcommand{\ul}{\underline}

\newcommand{\vp}{\varphi}
\def\R{{\mathbb R}}

\newcommand{\sub}{\ensuremath{ \subseteq}} 
\newcommand{\set}[1]{\ensuremath{ \left\{#1\right\} }} 
\newcommand{\norm}[1]{\ensuremath{ \left\lVert#1\right\rVert }} 
\newcommand{\absval}[1]{\ensuremath{ \left\lvert#1\right\rvert }} 
\newcommand{\spot}{\ensuremath{ \makebox[1ex]{\textbf{$\cdot$}} }}
\DeclareMathOperator{\supp}{supp}
\DeclareMathOperator{\Card}{Card}

\newcommand{\Gm}{\ensuremath{\Gamma}}

\newcommand{\om}{\ensuremath{\omega}}

\newcommand{\bt}{\ensuremath{\beta}}

\newcommand{\kp}{\ensuremath{\kappa}}

\newcommand{\ta}{\ensuremath{\theta}}

\begin{document}

\title{Quantization dimension for infinite conformal iterated function systems}

\author{ Jason Atnip}
\address{School of Mathematics and Statistics\\
	University of New South Wales\\
	Sydney NSW 2052, Australia}
\email{j.atnip@unsw.edu.au}

\author{ Mrinal Kanti Roychowdhury}
\address{School of Mathematical and Statistical Sciences\\
University of Texas Rio Grande Valley\\
1201 West University Drive\\
Edinburg, TX 78539-2999, USA.}
\email{mrinal.roychowdhury@utrgv.edu}

\author{ Mariusz Urba\'nski}
\address{Department of Mathematics\\
University of North Texas\\
Denton, TX 76203-1430, USA.}
\email{urbanski@unt.edu}

\thanks{The research of the second author was supported by U.S. National Security Agency (NSA) Grant H98230-14-1-0320. The research of the third author was supported in part by the
NSF Grant DMS 0700831.}

\subjclass[2000]{Primary 60Exx; Secondary 37A50, 94A34, 28A80.}
\keywords{Quantization dimension; F-conformal measure; topological
  pressure; temperature function.}

\date{}
\maketitle

\pagestyle{myheadings}\markboth{Jason Atnip, Mrinal Kanti Roychowdhury, and  Mariusz Urba\'nski}{Quantization dimension for infinite conformal iterated function systems}
\begin{abstract}
The quantization dimension function for an $F$-conformal measure $m_F$
generated by an infinite conformal iterated function system satisfying
the strong open set condition and by a summable H\"{o}lder family of
functions is expressed by a simple formula involving the temperature function of
the system. The temperature function is commonly used to perform the
multifractal analysis, in our context of the measure $m_F$. The result
in this paper extends a similar result of
Lindsay and Mauldin established for finite conformal iterated function systems [Nonlinearity 15 (2002)].
\end{abstract}

\section{Introduction}
Various types of dimensions such as Hausdorff and packing dimensions or
the lower and the upper box-counting dimensions are important to
characterize the complexity of highly irregular sets. In the past
decades, a lot of research has been done aiming at the calculation of
these dimensions for various special cases or establishing some
significant properties. In recent years, paralleling methods have been
adopted to study the corresponding dimensions for measures
(see \cite{F}). In this paper, we study the quantization dimension for
probability measure. The quantization problem consists in studying
the quantization error
induced by the approximation of a given probability measure with
discrete probability measures of finite supports. This problem
originated in information theory and some engineering technology. A
detailed account of this theory can be found in \cite{GL1}. Given a
Borel probability measure $\mu$ on $\D R^d$, a number $r \in (0,
+\infty)$ and a natural  number $n \in \D N$, the $n^{th}$
\tit{quantization error of order $r$} for $\mu$ is defined by
\begin{align*}
	V_{n, r}(\mu):=\te{inf}\set{\int d(x, \ga)^r d\mu(x): \ga \sci \D
  R^d, \, \Card(\ga) \leq n},
\end{align*}
where $d(x, \ga)$ denotes the distance from the point $x$ to the set
$\ga$ with respect to a given norm on $\D R^d$. Letting $|x|$ denote the Euclidean norm for $x\in\R^d$, we note
that if $\int | x|^r d\mu(x)<\infty$, then there is some set $\ga$
for which the infimum is achieved (see \cite{GL1}). Such a set $\ga$ for which the infimum occurs and contains no more than $n$ points is called an \tit{optimal set of $n$-means} of order $r$ for $0<r<+\infty$. To see some work in the direction of optimal sets of $n$-means, one is refereed to \cite{DR, GL4, R2, R3, R4, RR}.
The \tit{upper} and the \tit{lower
quantization dimensions} of order $r$ of $\mu$ are defined to be
\[
	\ol D_r(\mu): =\limsup_{n \to \infty} \frac{r\log n}{-\log V_{n, r}(\mu)}; \ \ul D_r(\mu): =\liminf_{n \to \infty} \frac{r\log n}{-\log V_{n, r}(\mu)}.
\]
If $\ol D_r(\mu)$ and $\ul D_r(\mu)$ coincide, we call the common value the \tit{quantization dimension} of order $r$ of the probability measure $\mu$, and is denoted by $D_r:=D_r(\mu)$. One sees that the quantization dimension is actually a function $r\mapsto D_r$ which measures the asymptotic rate at which $V_{n, r}$ goes to zero. If $D_r$ exists, then one can write
\[
	\log V_{n, r} \sim \log \left[\left(1/n\right)^{r/D_r}\right].
\]
For $s>0$, we define the $s$-dimensional \tit{upper} and \tit{lower} quantization coefficients for $\mu$ of order $r$ by $\limsup_{n\to\infty} n V_{n,r}^{s/r}(\mu)$ and $\liminf_{n\to\infty} n V_{n,r}^{s/r}(\mu)$, respectively. Compared to the lower and the upper quantization dimensions, the lower and the upper quantization coefficients provide us with more accurate information on the asymptotic of the quantization error.
Graf and Luschgy first determined the quantization dimension of order $r$ of a probability measure generated by a
finite system of self-similar mappings associated with a probability
vector (see \cite{GL1, GL2}). Lindsay and Mauldin extended the above
result to the $F$-conformal measure $m$ associated with a conformal
iterated function system determined by finitely many conformal
mappings (see \cite{LM}). Both the above results also show that quantization dimension function has a relationship with the temperature function of the thermodynamic formalism that arises in multifractal analysis of the measure. In \cite{MR}, Mihailescu and Roychowdhury determined the quantization dimension of probability measures generated by infinite system of self-similar mappings satisfying an strong open set condition associated with probability vectors, which is an infinite extension of the result of Graf-Luschgy (see \cite{GL1, GL2}). In this paper, we give an
extension of Lindsay and Mauldin's (see \cite{LM}) result to the realm
of iterated function systems with a countably infinite alphabet. The probability
measure $m_F$ considered here is the $F$-conformal measure associated
with a summable H\"older family of functions
\begin{align*}
	F:=\set{f^{(i)} : X \to \D R, i \in I}	
\end{align*}
and a conformal
iterated function system
\begin{align*}
	\Phi=\set{\vp_i : X \to X, i \in I}	
\end{align*}
where $I$
is a countable set, called alphabet, with finitely many, or what we want to emphasize,
infinitely many, elements. We show that for this measure $m_F$, the
quantization dimension function $D_r$, $0<r<+\infty$, exists and is uniquely determined by the following formula:
\begin{equation} \label{eq200011} \lim_{n\to\infty}\frac 1 n\log
  \sum_{\go \in
    I^n}\big(\|\exp(S_\go(F))\|\|\vp_\go'\|^r\big)^{\frac{D_r}{r+D_r}}=0,
\end{equation}
where $\norm{\spot}$ denotes the supremum norm on $X$.
The multifractal formalism for a probability measure corresponding to a
two parameter family of H\"older continuous functions
\begin{align*}
G_{q,t}:=\set{g_{q, t}^{(i)} :=q f^{(i)} +t \log |\vp_i'| : i \in I}
\end{align*}
does indeed hold if $I$ is a countable set (see \cite{HMU, MU}). In particular, the
singularity exponent $\gb(q)$ (also known as the temperature function)
satisfies the usual equation
\begin{equation}\label{eq2001}
	\lim_{n\to\infty}\frac 1 n\log
	\sum_{\go \in I^n}\|\exp(S_\go(F))\|^q
	 \|\vp_\go'\|^{\gb(q)}=0,
\end{equation}
and that the spectrum $f(\ga)$ is the Legendre transform of
$\gb(q)$. Comparing \eqref{eq200011} and \eqref{eq2001}, we see that if
$q_r=\frac{D_r}{r+D_r}$, then $\gb(q_r)=rq_r$, that is, the
quantization dimension function of order $r$ of an infinite $F$-conformal measure
has a relationship with the temperature function of the thermodynamic
formalism arising in multifractal analysis. For thermodynamic
formalism, multifractal analysis and the Legendre transform one can
see \cite{F, HMU}.

\section{Basic definitions and lemmas}
In this paper, $\D R^d$ denotes the $d$-dimensional Euclidean space equipped with a metric $d$ compatible with the Euclidean topology. Let us write,
\begin{align*}
	 V_{n, r}=V_{n, r}(\mu):=\inf\set{\int d(x, \ga)^r d\mu(x) : \ga \sci \D R^d, \ Card(\ga)\leq n}, \te{ and then } e_{n, r}(\mu):=V_{n, r}^{\frac 1 r}(\mu).
\end{align*}
Let us set
\begin{align*}
	u_{n, r} (\mu)=\inf\set{\int d(x, \ga\uu U^c)^r d\mu(x) : \ga \sci \D R^d, \Card(\ga)\leq n},
\end{align*}
where $U$ is the set which comes from the strong open set condition (definition follows) and $U^c$ denotes the complement of $U$. We see that
\[
	u_{n, r}^{\frac 1 r} \leq V_{n, r}^{\frac 1 r}=e_{n, r}.
\]
We call sets $\ga_n\sci \D R^d$, for which the above infimum are achieved, \tit{$n$-optimal sets} for $e_{n,r}$, $V_{n, r}$ or $u_{n, r}$, respectively.  As stated before, $n$-optimal sets exist when $\int |x|^r d\mu(x)<\infty$.

Let $X$ be a nonempty compact subset of $\D R^d$ with $\te{cl(int}(X))=X$ and $I$ be a countable set with infinitely many elements. Without any loss of generality we can take $I=\set{1, 2,  \cdots}$, i.e., the set of natural numbers. Let $\Phi=\set{\vp_i : i \in I}$ be a collection of injective contractions from $X$ into $X$ for which there exists $0<s<1$ such that
\begin{equation} \label{eq112} d(\vp_i(x), \vp_i(y))\leq s d(x, y)\end{equation}
for every $i \in I$ and every pair of points $x, y \in X$. Thus, the system $\Phi$ is uniformly contractive. Any such collection $\Phi$ of contractions is called an iterated function system.  By $\go:=\go_1\go_2\cdots \go_n \in I^n$, it is meant that $\go$ is a word of length $n$ over the symbols in $I$, $n\geq 1$, the length of an empty word is zero. Sometimes, we denote the length of a word $\go \in I^\ast$ by $|\go|$, and by $\go^-$ we denote the word obtained from $\go$ by deleting the last letter of $\go$, i.e., $
	\go^-=\go_1\go_2\cdots \go_{n-1}
$
if $\go=\go_1\go_2\cdots \go_n$ for $n\geq 1$. Write $I^*:=\UU_{n\geq 1}I^n$ to denote the set of all finite words in $I$. For $\go=\go_1\go_2\cdots \go_n \in I^n$, set
$\vp_\go=\vp_{\go_1}\circ \vp_{\go_2}\circ \cdots \circ \vp_{\go_n}.$
We have made the convention that the empty word $\es$ is the only word of length $0$ and $\vp_\es=\te{Id}_X$.
If $\go \in I^\ast\uu I^\infty$ and $n\geq 1$ does not exceed the length of $\go$, we denote by $\go|_n$ the word $\go_1\go_2\cdots \go_n$. Observe that given $\go \in I^\infty$, the compact sets $\vp_{\go|_n}(X)$, $n\geq 1$, are decreasing and their diameters converge to zero. In fact by \eqref{eq112}, we have
\begin{equation} \label{eq113}
	\te{diam}(\vp_{\go|_n}(X))\leq s^n \te{diam}(X),
\end{equation}
which implies that the set
\[
	\pi(\go)=\II_{n=1}^\infty \vp_{\go|_n}(X)
\]
is a singleton, and therefore, this formula defines a map $\pi : I^\infty \to X$ which, in view of \eqref{eq113} is continuous. The main object of our interest will be the limit set
\[
	J:=\pi(I^\infty)=\UU_{\go \in I^\infty}\II_{n= 1}^\infty \vp_{\go|_n}(X).
\]
The limit set $J$ is not necessarily compact. Let $\gs : I^\infty \to I^\infty$ denote the left shift map (cutting out the first coordinate) on $I^\infty$, i.e., $\gs(\go)=\go_2\go_3\cdots$ where $\go=\go_1\go_2\cdots$. Note that
$\pi\circ \gs(\go)=\vp_{\go_1}^{-1}\circ \pi(\go)$, and hence, rewriting $\pi(\go)=\vp_{\go_1}(\pi(\gs(\go)))$, we see that
\[
	J=\UU_{j=1}^\infty \vp_j(J).
\]
In the sequel, for $\go \in I^\ast$ we write $J_\go:=\vp_\go(J)$. We say that the iterated function system satisfies the \tit{open set condition} (OSC) if there exists a bounded nonempty open set $U \sci X$ (in the topology of $X$) such that $\vp_i(U) \sci U$ for every $i \in I$ and $\vp_i(U) \ii \vp_j(U)=\es$ for every pair $i, j \in I, \, i\neq j$, and the \tit{strong open set condition} (SOSC) if $U$ can be chosen such that $U \ii J \neq \es$. We assume that the infinite iterated function system considered in this paper satisfies the strong open set condition. An iterated function system satisfying the open set condition is said to be \te{conformal} if the following conditions are satisfied:
\begin{enumerate}[(i)]
	
\item $U=\te{Int}_{\D R^d}(X)$.

\item There exists an open connected set $V$ with $X \sci V \sci \D R^d$ such that all maps $\vp_i$, $i\in I$, extend to $\C C^1$-conformal diffeomorphisms of $V$ into $V$.

\item There exist $\gg, \ell >0$ such that for every $x \in \partial X \sci \D R^d$ there exists an open cone Con$(x, \gg, \ell) \sci \te{Int}(X)$ with vertex $x$, central angle of Lebesgue measure $\gg$, and altitude $\ell$.

\item Bounded Distortion Property (BDP): There exists a constant $K\geq 1$ such that
\[
	|\vp_{\go}'(y)| \leq K|\vp_{\go}'(x)|
\]
for every $\go \in I^*$ and every pair of points $x, y \in V$, where $|\vp_\go'(x)|$ means the norm of the derivative.
\end{enumerate}
For the conformal iterated function system let us now state the
following well-known lemma (for details of the proof see
\cite{P}).

\begin{lemma} \label{lemma5566} There exists a constant $\tilde K \geq K$ such that
\[
	\tilde K^{-1}\|\vp_\go'\|d(x, y)\leq d(\vp_\go(x), \vp_\go(y))\leq \tilde K \|\vp_\go'\|d(x, y)
\]
for every $\go\in  I^\ast$ and every pair of points $x, y \in V$, where $d$ is the metric on $X$.
\end{lemma}
Inequality \eqref{eq112} implies that for every $i\in I$,
\[
	\|\vp_i'\|=\sup_{x \in X} |\vp_i'(x)|= \sup_{x\in X}  \lim_{y\to x} \frac{d(\vp_i(y), \vp_i(x))}{d(y, x)} \leq  \sup_{x\in X} \lim_{y\to x}\frac{s d(x, y)}{d(x, y)}=s,
\]
and hence, $\|\vp_\go'\|\leq s^n$ for every $\go \in I^n$, $n\geq
1$. For $t\geq 0$, the topological pressure function of the conformal
iterated function system $\Phi=\set{\vp_i : i \in I}$ is given by
\[
	P(t)=\lim_{n\to \infty} \frac  1 n \log \sum_{\go \in I^n} \|\vp_\go'\|^t,
\]
provided the limit exists.
Define
\begin{align*}
	\ta_\Phi:=\inf\set{t\geq 0:P(t)<\infty}.	
\end{align*}
the following Proposition of \cite{MU1} describes the good behavior of the topological pressure function.
\begin{prop}(see \cite{MU1} Proposition 3.3)
	$P(t)$ is non-increasing on $[0,\infty)$ and strictly decreasing, convex, and continuous on $(\ta_\Phi,\infty)$ with $P(t)\to-\infty$ as $t\to\infty$.
\end{prop}

As it was shown in \cite{MU1}, there are
two disjoint classes of conformal iterated function systems, regular
and irregular. A system is regular if there exists $t\geq 0$ such that
$P(t)=0$. Otherwise, the system is irregular. Moreover, if $\Phi$ is a
conformal iterated function system, then Theorem 3.15 of \cite{MU1} gives
\begin{align*}
	\te{dim}_{\te{H}}(J)=\sup\set{\te{dim}_{\te{H}} (J_F) : F \sci I, F \te{
    finite}}=\inf \set{ t \geq 0 : P(t)\leq 0}\geq \ta_\Phi,
\end{align*}
where $\te{dim}_{\te{H}}(J)$ represents the Hausdorff dimension of the limit
set $J$, and \ $J_F$ is the limit set associated to the index set
$F$. If a system is regular and $P(t)=0$, then $t=\te{dim}_{\te{H}}(J)$. Let
us assume that the conformal iterated function system considered in
this paper is regular.

Let $F=\set{f^{(i)} : X\to \D R}_{i\in I}$ be a family of continuous functions such that for each $n\geq 1$ if we define
\[
	v_n(F)=\sup_{\go \in I^n}\sup_{x, y \in X}\left\{|f^{(\go_1)}(\vp_{\gs(\go)}(x))-f^{(\go_1)}(\vp_{\gs(\go)}(y))|\right\}e^{\gb(n-1)}
\]
for some $\gb>0$, then the following is satisfied:
\begin{equation} \label{eq72}
	v_{\gb}(F)=\sup_{n\geq 1}\left\{v_n(F)\right\}<\infty.
\end{equation}
The collection $F$ is called then a \tit{H\"older family of functions} (of order $\gb$). Denote by $\|\spot\|$ the supremum norm on the Banach space $\C C(X)$, and by $1\!\!1$ the function with constant value 1 on $X$. If in addition to \eqref{eq72} we have
\[
	\sum_{i \in I} \|e^{f^{(i)}}\|<\infty \te{ or  equivalently } \C L_F(1\!\!1) \in \C C(X),
\]
where
\[
	\C L_F(g)(x)=\sum_{i \in I} e^{f^{(i)}(x)}g(\vp_i(x)), \quad  g\in \C C (X),
\]
is the associated Perron-Frobenius or transfer operator, then $F$ is called a \tit{summable H\"older family of functions} (of order $\gb$). It
was originally in \cite{HMU}, and called a strongly H\"older family of functions. In our paper, we assume that $F$ is a  summable H\"older
family of functions of order $\gb$. For $n\geq 1$ and $\go \in I^n$, set
\begin{align*}
	S_\go(F):=\sum_{j=1}^nf^{(\go_j)}\circ\vp_{\gs^j(\go)}.
\end{align*}
Then, following the classical thermodynamic formalism, the \tit{topological pressure} of $F$ is defined by
\[
	P(F):=\lim_{n\to \infty} \frac 1 n\log \sum_{\go \in I^n}\|\exp(S_\go(F))\|.
\]
The limit above exists by the standard theory of subadditive sequences. Subtracting from each of the functions $f^{(i)}$ the topological pressure of $F$ we may assume that $P(F)=0$. By \cite{HMU}, there exists
a unique Borel probability measure $m_F$ on
$X$ with $m_F(J)=1$ such that for any continuous function
  $g : X\to \D R$ and $n\geq 1$,
\begin{equation*}
	\int g\,dm_F=\sum_{\go\in I^n}\int \exp(S_\go(F))\cdot (g\circ \vp_\go)\,dm_F.
\end{equation*}
In particular, for any Borel set $A \sci J$ and $\gt \in I^n$, $n\geq 1$, we have
\begin{align*}
	m_F(\vp_\gt (A))& =\sum_{\go \in I^n}\int
  \exp(S_\go(F)(x))\cdot (I_{\vp_\gt(A)}\circ \vp_\go(x))\,dm_F(x)\\
	&=\int \exp(S_\gt(F)(x))\cdot (I_{\vp_\gt(A)}\circ \vp_\gt(x))\,dm_F(x)\\
	&=\int_A \exp(S_\gt(F)(x))\,dm_F(x).
\end{align*}
Moreover, $m_F$ satisfies $m_F(\vp_\go (X) \ii \vp_\gt(X))=0$ for all
incomparable words $\go, \gt \in I^*$. The probability measure
$m_F$ is called the \tit{$F$-conformal measure} of the infinite conformal iterated
function system $\Phi$ and the summable H\"older family of functions
$F=\set{f^{(i)} : X \to \D R, i \in I}$.

The following lemma is known.

\begin{lemma} (see \cite{LM}, Lemma 2) \label{lemma92}  There exists a constant $C\geq 1$ such that for any $x, y \in X$ and $\go \in I^*,$
\[
	\frac{\exp(S_\go(F)(x))}{\exp(S_\go(F)(y))} \leq C.
\]
In particular, for any $x \in X$ and $\go \in I^*$,
\begin{align*}
	\exp(S_\go(F)(x))\geq C^{-1}\|\exp(S_\go(F))\|.
\end{align*}
\end{lemma}
Using Lemma~\ref{lemma92}, we can deduce the following lemma.

\begin{lemma} (see \cite[Lemma~3.3]{R1}) \label{lemma921} Let $C\geq 1$ be as in Lemma~\ref{lemma92}. Then, for $\go, \gt \in I^\ast$, and $x, y \in X$, we have
\[
	\exp(S_{\go\gt}(F)(x)) \geq C^{-2} \|\exp(S_{\go}(F))\|\|\exp(S_{\gt}(F))\|.
\]
\end{lemma}
Let us now consider a two-parameter family of H\"{o}lder continuous functions
\begin{align*}
	G_{q, t}=\set{g^{(i)}(q, t):=q f^{(i)} + t\log |\vp_i'|}_{i \in I}.
\end{align*}
The \tit{topological pressure} corresponding to $G_{q, t}$ is given by
\begin{equation} \label{eq1}
	P(q, t)=\lim_{n\to \infty} \frac 1 n\log\sum_{\go \in I^n}\|\exp(S_\go(F))\|^q\|\vp_\go'\|^t.
\end{equation}
The limit above exists by the standard theory of subadditive
sequences. For $q=0$ this simply means that the system $\Phi$ is strongly regular;
see \cite{MU} for a detailed discussion of this concept. Let
\begin{align*}
	\te{Fin}(q)=\set{t \in \D R : \C L_{G_{q, t}}(1\!\!1) <\infty}
	=\set{ t \in \D R : P(q, t)<\infty} \te{ and } \gq(q)=\te{inf Fin}(q).
\end{align*}
Notice that either $\te{Fin}(q)=(\theta(q),+\infty)$, or $\te{Fin}(q)=[\theta(q),+\infty)$. We assume that for every $q\in [0,1]$ there exists $u\in (\gq(q), +\infty)$ such that
\begin{equation}\label{eqmu1}
0<P(q,u)<+\infty.
\end{equation}

The following lemma is easy to prove (see \cite[Lemma 7.1]{HMU}).

\begin{lemma}
For every $q\in \D R$ the function $(\theta(q),+\infty)\ni t\mapsto P(q,
t)$ is strictly decreasing, convex and hence continuous, and $P(q,t)\to-\infty$ as $t\to\infty$.
\end{lemma}

\noindent With the use of \eqref{eqmu1}, the proof of \cite[Lemma 7.2]{HMU} gives
the following.

\begin{lemma}
If $q \in [0,1]$, then there exists a unique $t=\gb(q)\in (\gq(q), +\infty)$ such
that $P(q, \gb(q))=0$.
\end{lemma}

\noindent By \cite[Theorem 7.4]{HMU}, the function $\gb$ is strictly
decreasing, convex and hence continuous on $[0,1]$. This
function is commonly called the \tit{temperature function} of the
thermodynamic formalism under consideration.

\begin{note}
Since the system $\Phi$ is strongly regular $\gb(0)=\te{dim}_\te{H}(J)$, where  $\te{dim}_\te{H}(J)$ denotes the Hausdorff dimension of the
limit set $J$ (see \cite{HMU}). Moreover, $P(1, 0)=0$, which gives
$\gb(1)=0$.
\end{note}

In the next sections we state and prove the main result of the paper. In addition, we also give some connections between the quantization processes of the infinite conformal iterated function systems and its truncated systems.

\section{Main result}

Consider the function $g : (0, 1] \to \D R$ defined, for an arbitrary $r \in (0, \infty)$,  by the following formula:
$$
g(x)=\frac{\gb(x)}{rx}.
$$
We know that $\gb(1)=0$ and
  $\gb(0)=\te{dim}_\te{H}(J)$, and so $g(1)=0$ and $\lim_{x\to 0+}
  g(x)=+\infty$. Moreover, the function $g$ is continuous, even
  real-analytic (see \cite{HMU} for the function $\beta$), and strictly decreasing (calculate its derivative
  which is negative since $\beta'<0$)
   on $(0, 1]$. Hence, there exists a unique $q_r \in (0,
  1)$ such that $g(q_r)=1$, i.e.,
$$
\gb(q_r)=rq_r.
$$

The relationship between the quantization dimension function and the
temperature function $\gb(q)$ for the $F$-conformal measure $m_F$, where
the temperature function is the Legendre transform of the $f(\ga)$
curve (for the definitions of $f(\ga)$ and the Legendre transform see
\cite{F, HMU}) is given by the following theorem which constitutes the
main result of this paper. For a graphical description see Figure~\ref{Fig1}.

\begin{theorem} \label{theorem}
Let $m_F$ be the $F$-conformal measure associated with the
infinite family of strongly H\"{o}lder  functions $F:=\set{f^{(i)} : X
  \to \D R}_{i\in I}$ and the
infinite conformal iterated function system $\Phi:=\set{\vp_i : i\in I}$ satisfying the strong open set condition. Let $P(q, t)$ be the corresponding topological pressure such that for each $q\in [0,1]$ there exists $u\in (\gq(q), +\infty)$ with $0<P(q, u)<+\infty$. Then, for each $r
\in (0, +\infty)$ the
quantization dimension (of order $r$) of the probability measure $m_F$
is given by
\[
	D_r(m_F)=\frac{\gb(q_r)}{1-q_r},
\]
where, we recall $\beta$ is the temperature function.
\end{theorem}

\begin{figure}
\includegraphics[width=4in]{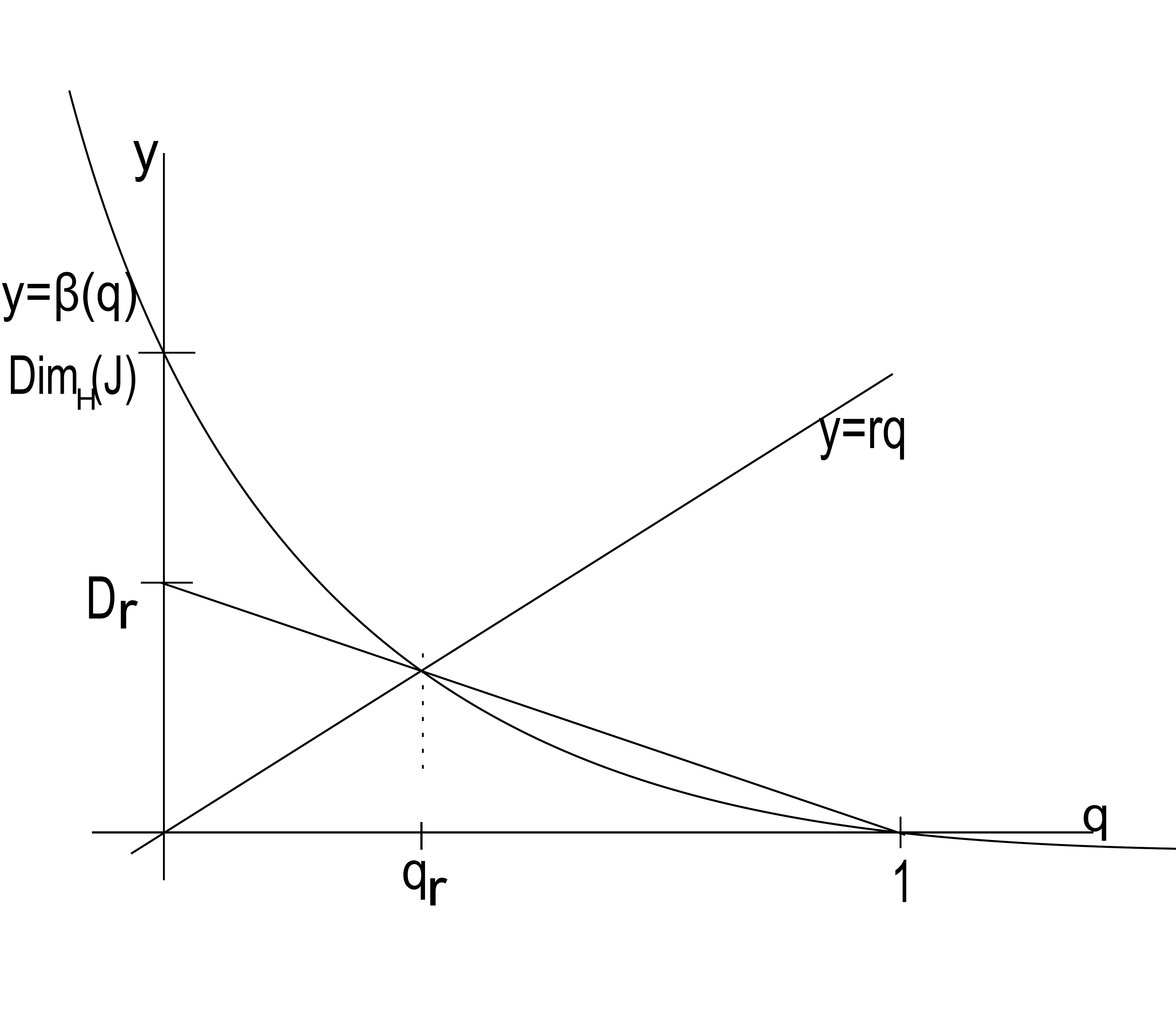}
\caption {\label{Fig1.}
To determine $D_r$ first find the point of intersection of $y=\gb(q)$
and the line $y=rq$. Then, $D_r$ is the $y$-intercept of the line
through this point and the point (1, 0).} \label{Fig1}
\end{figure}

To prove the theorem we need to prove some lemmas.
\begin{lemma} \label{lemma101}
Let $0<r<+\infty$ be given. Then, there exists exactly one number
$\gk_r \in (0, +\infty)$ such that
\[
	P\left(\frac{\gk_r}{r+\gk_r}, \frac{r\gk_r}{r+\gk_r}\right)=0.
\]
\end{lemma}
\begin{proof} Just take $\gk_r:=\frac{rq_r}{1-q_r}$ and apply the
definition of $q_r$.
\end{proof}

\begin{remark} \label{rem889}
We postpone the proof of Theorem \ref{theorem} for Section 4, and now give some approximation results for the temperature and quantization dimension in the infinite case by using \textit{truncated finite systems}.

For arbitrary $M\geq 2$ write $I_M=\set{1, 2, \cdots, M}$, and consider
the partial iterated function system
\begin{align*}
	\Phi_M=\set{\vp_i : X \to X : i \in I_M}
\end{align*}
of the infinite system $\Phi$. Let $J_M$ be its limit set. Consider also the partial H\"older
family of functions
\begin{align*}
	F_M=\set{f^{(i)} : X \to \D R : i \in I_M}
\end{align*}
of the infinite family $F$. Let $m_M$ be the corresponding
$F_M$-conformal measure on $J_M$. Furthermore, let $P_M(q, t)$ be the
topological pressure given by
\begin{align}\label{j eq sub pressure def}
P_M(q,t):=\lim_{n\to \infty} \frac 1 n\log\sum_{\go \in I_M^n}\|\exp(S_\go(F))\|^q\|\vp_\go'\|^t
\end{align}
and $\beta_M(q)$ be the
temperature function associated with the system $\Phi_M$, that is,
\begin{align*}
	P_M(q,\bt_M(q))=0.
\end{align*}
Note that for each $M\geq 2$, $P_M(q, t)$ is strictly decreasing, convex and hence continuous in each variable $q, t \in \D R$ separately and $P_M(q,t)\to-\infty$ as $t\to\infty$ (see [F1, P]). Furthermore, we note that
\begin{align*}
	P_M(q,t)\leq P(q,t).
\end{align*}
\end{remark}
The following lemma is a special case of Lemma~\ref{lemma101}.
\begin{lemma} \label{lemma102}
Let $0<r<+\infty$ and $M\ge 2$ be as before. Then, there exists exactly
one number $\gk_{r, M} \in (0, +\infty)$ such that
\[
	P_M\left(\frac{\gk_{r, M}}{r+\gk_{r, M}}, \frac{r\gk_{r, M}}{r+\gk_{r, M}}\right)=0.	
\]
\end{lemma}
Lindsay and Mauldin showed that the above $\gk_{r, M}$ is the
quantization dimension of order $r$ of the probability measure $m_M$
(see \cite{LM}). We shall prove the following lemma.

\begin{lemma} \label{lemma104}
Let $(q_M)_{M\geq 2}$ be a sequence of elements in $[0, 1]$ such that
$q_M \to q$ for some $q \in [0, 1]$. Then,
$
	\gb_M(q_M) \to \gb(q)
$
as $M\to \infty$.
\end{lemma}

\begin{proof}
First observe that the sequence of functions $(\beta_M)_{M=2}^\infty$
is increasing on the compact space $[0,1]$ and that
$\beta_M(q)\le\beta(q)$ for all $q\in[0,1]$. Denote
\begin{align*}
	\hat\beta(q):=\lim_{M\to\infty}\beta_M(q).	
\end{align*}
Then, we have
$
\hat\beta(q)\le \beta(q)
$
for all $q\in[0,1]$. But,
$$
0=  P(q,\beta(q))
\le P(q,\hat\beta(q))
=   \lim_{M\to\infty} P_M(q,\hat\beta(q))
\le \lim_{M\to\infty} P_M(q,\beta_M(q))
=0.
$$
Hence, $P(q,\hat\beta(q))=0$, and therefore
$
\hat\beta(q)=\beta(q).
$
It now follows from Dini's Lemma that the sequence
$(\beta_M)_{M=2}^\infty$ converges to $\beta$ uniformly on
$[0,1]$. Moreover,  the temperature functions are continuous on $[0, 1]$. Therefore, it follows that if $q_M \to q$, then $\beta(q)=\lim_{M\to\infty}\beta_M(q_M)$.
\end{proof}

%
%

\begin{lemma} \label{lemma42}
Let $0<r<+\infty$, and let $\gk_r$ and $\gk_{r, M}$ be as in Lemma~\ref{lemma101} and Lemma~\ref{lemma102}. Then,
$
	\gk_{r, M} \to \gk_r
$
as $M \to \infty$.
\end{lemma}

\begin{proof}
Let $q_{r, M}=\frac{\gk_{r, M}}{r+\gk_{r, M}}$ and
$q_r=\frac{\gk_r}{r+\gk_r}$. It is enough to prove $q_{r, M} \to q_r$
as $M\to \infty$. Let $(q_{r, M_K})_{k\geq 1}$ be a subsequence of
$(q_{r, M})_{M\geq 2}$ such that $q_{r, M_k} \to \hat q$ for some
$\hat q \in [0, 1]$. Then, by Lemma~\ref{lemma104},
\[
	r\hat q =\lim_{k\to \infty} r q_{r, M_k} =\lim_{k\to \infty} \gb_{M_k}(q_{r, M_k})=\gb(\hat q).
\]
So, by uniqueness of $q_r$, this implies that $\hat q=q_r$. Hence,
$q_{r, M} \to q_r$ as $M\to \infty$, i.e., $\gk_{r, M} \to \gk_r$ as
$M \to \infty$.
\end{proof}

\begin{lemma} Let $m_M$ and $m_F$ be as defined before. Then, $m_M$
  converges weakly to $m_F$ as $M\to \infty$.
\end{lemma}

\begin{proof}
Let $f:I^\infty\to I^\infty$ be such that
\begin{align*}
	f(\go)=f^{(\go_1)}(\pi(\gs(\go))),
\end{align*}
i.e., $f$ is an amalgamated
function corresponding to the infinite family $F$ of strongly H\"older
functions. Then, for each $M\ge 2$, $f_M:=f|_{I_M^\infty}$ is an amalgamated
function for the partial family $F_M$ of strongly H\"older  functions.  Let $\tilde m_M$ be the
conformal measure of the function $f_M$ with respect to the dynamical
system generated by the shift map $\gs_M : I_M^\infty\to
I_M^\infty$. Then, (see Theorem~3.2.3 in \cite{MU})
\begin{align*}
	m_M:=\tilde
	m_M\circ \pi_M^{-1}
\end{align*}
is the unique $F_M$-conformal measure, where $\pi_M:=\pi|_{I_M^\infty}$ is the restriction of the coding map $\pi$ on the space $I_M^\infty$. The proof of Theorem~2.7.3
along with Corollary~2.7.5 (especially its uniqueness part (a)) in \cite{MU} gives
that the sequence $(\tilde m_M)_{M=2}^\infty$ converges weakly to $\tilde
m$, the unique conformal measure for $f:I^\infty\to\R$. Since the
projection map $\pi:I^\infty\to X$ is continuous, we therefore have
that the sequence $(m_M)_{M=2}^\infty$ converges weakly to a
measure $m$ given by
\begin{align*}
	m:=\tilde m\circ \pi^{-1}.	
\end{align*}
But because of Theorem~3.2.3 in \cite{MU}
again, so defined measure
$m$ is the unique conformal measure $m_F$ for the H\"older family $F$,
i.e., $m=m_F$.
Thus, the proof of the lemma is complete.
\end{proof}
Let $\C M$ denote the set of all Borel probability measures on $X$, and $0<r<+\infty$ be fixed. Since $X$ is compact, for any Borel probability measure $\gn$ on $X$ we have $\int|x|^rd\gn(x)<\infty$. Then, we know $L_r$-minimal metric (also refereed as $L_r$-Wasserstein metric or $L_r$-Kantorovich metric) is given by
\[
	\rho_r(P_1, P_2)=\inf_\gn \left(\int |x-y|^r \,d\gn(x, y)\right)^{\frac 1 r},
\]
where the infimum is taken over all Borel probabilities $\gn$ on $X \times X$ with fixed marginals $P_1$ and $P_2$.
Again we know that in the weak topology on $\C M$,
\[
	m_M \to m_F  \Leftrightarrow \int_X f \,dm_M - \int_X f\, dm_F \to 0 \te{ for all } f \in \C C(X),
\]
where $\C C (X) : =\set{ f : X \to \D R : f \te{ is continuous }}$.
Note that weak topology and the topology induced by $L_r$-minimal metric $\rho_r$ coincide on $\C M$ (see \cite{Ru}).

Let us now state the following lemma.

\begin{lemma} (see \cite[Lemma~3.4]{GL1})  \label{lemma345}
Let $\C P_n$ denote the set of all discrete probability measures $Q$ on $X$ with $|\supp(Q)|\leq n$. Then, given $P\in\C M$ we have
\[
	V_{n, r}(P)=\inf_{Q\in \C P_n} \rho_r^r(P, Q).
\]
\end{lemma}

Let us now prove the following lemma.
\begin{lemma} \label{lemma43}
Let $0<r<+\infty$, and $m_M \to m_F$ with respect to the weak topology. Then, for all $n\geq 1$,
\[
	\lim_{M\to \infty} V_{n, r}(m_M)=V_{n, r}(m_F).
\]

\end{lemma}

\begin{proof}
Since $X$ is compact, for any Borel probability measure $\gn$ on $X$ we have $\int|x|^r\,d\gn(x)<\infty$. Hence, by Lemma~\ref{lemma345}, for any $n\geq 1$, it follows that
\begin{align*}
	\absval{V_{n, r}^{1/r}(m_M)-V_{n,r}^{1/r}(m_F)}\leq \rho_r(m_M, m_F),
\end{align*}
which yields the lemma.
\end{proof}
The following lemma is useful.

\begin{lemma}\label{lemma44}  For any two Borel probability measures $\mu$ and $\gn$ and for all sufficiently large $n$, if $V_{n, r}(\mu)\leq V_{n, r}(\gn)<1$, then 
	\[\ul D_r(\mu) \leq \ul D_r(\gn) \te{ and } V_{n, r}^{\frac{\ul D_r(\mu)}{r}} (\mu) \leq V_{n, r}^{\frac{\ul D_r(\gn)}{r}}(\gn),\]
where $\ul D_r(\mu)$ and $\ul D_r(\gn)$ are the lower quantization dimensions of order $r$ of $\mu$ and $\gn$, respectively.
\end{lemma}
\begin{proof}
Let $V_{n, r}(\mu)\leq V_{n, r}(\gn)<1$ for all sufficiently large $n$. Then, it follows that $\log V_{n, r}(\mu)\leq \log  V_{n, r}(\gn)<0$, which implies
\[\frac {1} {-\log V_{n, r}(\mu)} \leq \frac {1}{-\log V_{n, r}(\gn)}, \te{ and so }
\frac {r\log n} {-\log V_{n, r}(\mu)} \leq \frac {r\log n}{-\log V_{n, r}(\gn)}.\]
Now, taking liminf on both sides, we have $\ul D_r(\mu) \leq \ul D_r(\gn)$. Note that
\[\log V_{n, r}^{\frac{\ul D_r(\mu)}{r}} (\mu)=\frac{\ul D_r(\mu)}{r} \log  V_{n, r}(\mu)\leq \frac{\ul D_r(\gn)}{r} \log  V_{n, r}(\gn)=\log V_{n, r}^{\frac{\ul D_r(\gn)}{r}} (\gn),\]
and so $V_{n, r}^{\frac{\ul D_r(\mu)}{r}} (\mu) \leq V_{n, r}^{\frac{\ul D_r(\gn)}{r}}(\gn)$. Thus, the lemma is yielded.
\end{proof}
\begin{remark}
Similarly, under the above condition in Lemma~\ref{lemma44},  if $\ol D_r(\mu)$ and $\ol D_r(\gn)$ are the upper quantization dimensions of order $r$ of $\mu$ and $\gn$, respectively, then it can be proved that
\[\ol D_r(\mu) \leq \ol D_r(\gn) \te{ and } V_{n, r}^{\frac{\ol D_r(\mu)}{r}} (\mu) \leq V_{n, r}^{\frac{\ol D_r(\gn)}{r}}(\gn).\]

\end{remark}
We conclude this section with the following proposition which gives the desired upper bound for the quantization dimension.
\begin{prop} \label{prop33} Let $0<r<+\infty$ be fixed and $\gk_r$ be as in Lemma~\ref{lemma101}. Then,
	\begin{align*}
	\ul D_r(m_F)\geq\gk_r.
	\end{align*}
\end{prop}

\begin{proof}
	Let $m_M$ be the $F_M$ conformal measure for the truncated system as defined in Remark~\ref{rem889}. Then, we know $\gk_{r, M}$, given by Lemma~\ref{lemma102}, is the quantization dimension of the measure $m_M$. Since $\supp(m_M) \sci \supp(m_F)$, for all sufficiently large $n$, we have $V_{n, r}(m_M)\leq V_{n,r}(m_F)<1$, and then Lemma~\ref{lemma44} implies $\ul D_r(m_F)\geq \ul D_r(m_M)=\gk_{r, M}$. Now, take $M\to \infty$ and use Lemma~\ref{lemma44}, which yields $\ul D_r(m_F) \geq \gk_r$.
\end{proof}

\section{Proof of Theorem~\ref{theorem}}
In this section we prove our main theorem, Theorem~\ref{theorem}.

\begin{proof}[Proof of Theorem \ref{theorem}]

\noindent Let $\Phi=\set{\vp_i:X\to X}_{i\in I}$ be an infinite conformal iterated function system satisfying the strong open set condition and let $F=\set{f^{(i)} : X\to \D R}_{i\in I}$ be an infinite family of strongly H\"{o}lder functions. Let $\kp_r$ be the unique number given by Lemma~\ref{lemma101} with
\begin{align*}
	P\left(\frac{\kp_r}{r+\kp_r},\frac{r\kp_r}{r+\kp_r}\right)=0.
\end{align*}
Since $0<r<+\infty$ is fixed, write $q=\frac{\kp_r}{r+\kp_r}$. Then, by \cite{HMU}, there exists a unique Borel probability measure $m_q$ ($G_{q, \gb(q)}$-conformal measure) on $J$ such that for any continuous function $g : X \to \D R$ and $n\geq 1$, we have
\[
	\int g \,dm_q=\sum_{\go\in I^n} \int \Big(\exp(S_\go(F))|\vp_\go'|^r\Big)^q \cdot (g\circ \vp_\go) \,d m_q(x),
\]
in particular, for any Borel $A \sci J$, and $\go \in I^\ast$, we have
\begin{equation} \label{eq3451} m_q(\vp_\go(A))=\int_A \Big(\exp(S_\go(F))|\vp_\go'|^r\Big)^q \,dm_q(x).
\end{equation}
For $N\geq 2$, consider the family $F_N\sub F$ given by $
	F_N=\set{f^{(i)} : X\to \D R}_{i=1}^N
$
and the finite subsystem of $\Phi$ consisting of the first $N$ elements, that is,
\begin{align*}
	\Phi_N=\set{\tilde \vp_i:X\to X}_{i=1}^N,
\end{align*}
where the collection of maps $\set{\tilde\vp_1,\dots,\tilde\vp_N}$ is a permutation of the maps $\set{\vp_1,\dots,\vp_N}$ so that $\norm{\tilde{\vp}_i'}$ is decreasing for the relabeled maps, i.e.,
\begin{align*}
	\norm{\tilde{\vp}_i'}\geq\norm{\tilde{\vp}_{i+1}'}
\end{align*}
for $1\leq i<N$. Let $J_N$ be its associated limit set,  and let
$
	I_N=\set{1,\dots, N}.
$
Further, let $G_{N,q,t}$ be the subset of $G_{q,t}$ defined based on the subfamily $F_N$ of $F$, i.e.,
\begin{align*}
		G_{N,q,t}=\set{g^{(i)}(q, t):=q f^{(i)} + t\log |\tilde{\vp}_i'|}_{i \in I_N}.
\end{align*}
We call $\Gm \sub I^\ast$ a \tit{finite maximal antichain} if $\Gm$ is a finite set of words in $I^\ast$, such that every sequence in $I^\infty$ is an extension of some word in $\Gm$, but no word in $\Gm$ is an extension of another word in $\Gm$. In particular, we have that
\begin{align*}
	\tilde\vp_\om(J_N)\cap\tilde\vp_\tau(J_N)=\emptyset
\end{align*}
for all $\om\neq \tau \in \Gm\sub I_N^*$.
If $\om=\om_1\om_2\cdots\om_n \in I_N^n$, $n\geq 1$, set
\begin{align*}
	S_\go(F_N):=\sum_{j=1}^n f^{(\go_j)}\circ \tilde \vp_{\gs^j(\go)},	
\end{align*}
and we denote
\[
	\tilde \vp_\go:=\tilde \vp_{\go_1}\circ\cdots \circ \tilde \vp_{\go_n}, \te{ and } J_{N,\go}:=\tilde \vp_\go(J_N).
\]
Then, there exists a unique Borel probability measure $m_N$ (unique $F_N$-conformal measure) supported on $J_N$ such that for any continuous function
  $g : X\to \D R$ and $n\geq 1$,
\begin{equation*}
	\int g\,dm_N=\sum_{\go\in I_{N}^n}\int \exp(S_\om(F_N))\cdot (g\circ \tilde \vp_\go)\,dm_N.
\end{equation*}
Let $P_N(q, t)$ be the topological pressure as in \eqref{j eq sub pressure def} and let $\bt_N(q)$ be the temperature function in this case, i.e., for each $q\in \D R$ there exists a unique $\bt_N(q) \in \D R$ such that $P_N(q, \bt_N(q))=0$. Then, for each $q\in \D R$, there exists a probability measure $m_{N,q}$ ($G_{N,q,t}$-conformal measure) on $J_N$ such that for any continuous function $g: X \to \D R$ and $n\geq 1$,
\[
	\int g \,dm_{N,q}=\sum_{\go \in I_{N}^n} \int (\exp(S_\om(F_N))^q |\tilde \vp_\go'|^{\bt_N(q)}\cdot (g\circ \tilde \vp_\go)\, dm_{N,q}.
\]
In fact, for any Borel $A\sci J_N$ and $\go \in I_{N}^\ast$, we have
\begin{equation*}
	\label{eq34511} m_{N,q}(\tilde \vp_\go(A))=\int_A (\exp(S_\om(F_N))^q |\tilde \vp_\go'|^{\bt_N(q)} \,dm_{N,q}.
\end{equation*}
Again, for any $\go \in I_{N}^\ast$, using Bounded Distortion Property and Lemma~\ref{lemma92}, we have
\begin{align*}
	1\geq m_{N,q}(\tilde \vp_\go(J_N))& =\int_{J_{N,\om}} (\exp(S_\om(F_N))^q |\tilde \vp_\go'|^{\bt_N(q)} \,dm_{N,q}\\
	&\geq C^{-q}K^{-\bt_N(q)} \|\exp(S_\om(F_N))\|^q \|\tilde \vp_\go'\|^{\bt_N(q)},
\end{align*}
which for each $q\in \D R$ implies
\begin{align}\label{j eq 1}
	\|\exp(S_\om(F_N))\|^q \|\tilde \vp_\go'\|^{\bt_N(q)}\leq C^qK^{\bt_N(q)} m_{N,q}(\tilde \vp_\go(J_N)).
\end{align}
Then, for any finite maximal antichain $\gG\sci I_{N}^*$, it follows that
\begin{align*}
	\sum_{\om\in\Gm} m_{N,q}(J_{N,\om})\leq 1,
\end{align*}
which together with \eqref{j eq 1} gives
\begin{equation}
	\label{eq890}  \sum_{\go \in \gG} \|\exp(S_\om(F_N))\|^q\|\tilde \vp_\go'\|^{\bt_N(q)}\leq  C^qK^{\gb_N(q)}.
\end{equation}
As in Lemma~\ref{lemma101}, there exists exactly one number
$\kp_{r,N} \in (0, +\infty)$ such that
\begin{equation}
	\label{eq3333} P_N\Big(\frac{\kp_{r,N}}{r+\kp_{r,N}}, \frac{r\kp_{r,N}}{r+\kp_{r,N}}\Big)=0.
\end{equation}
Thus, if we take $q_N=\frac{\kp_{r,N}}{r+\kp_{r,N}}$, then $\bt_N(q_N)=rq_N$. Then, \eqref{eq890} implies
\begin{equation}
	\label{eq345} \sum_{\go \in \gG} \Big(\|\exp(S_\om(F_N))\|\|\tilde \vp_\go'\|^r\Big)^{\frac{\kp_{r,N}}{r+\kp_{r,N}}}\leq  (CK^r)^\frac{\kp_{r,N}}{r+\kp_{r,N}}.
\end{equation}
Now, since  $P_N(q,t)\leq P(q,t)$, and since the function
$
	x\longmapsto\frac{x}{x+r}
$
is increasing for $0<r<\infty$,  we have
\begin{align}\label{j eq 2}
	\frac{\kp_{r,N}}{\kp_{r,N}+r}\leq \frac{\kp_r}{\kp_r+r}.
\end{align}
In light of \eqref{eq345} and \eqref{j eq 2}, we can write
\begin{align}
	&\sum_{\go \in \gG} \Big(\|\exp(S_\om(F_N))\|\|\tilde \vp_\go'\|^r\Big)^{\frac{\gk_r}{r+\gk_r}}\nonumber\\
	&\quad= \sum_{\go \in \gG} \Big(\|\exp(S_\om(F_N))\|\|\tilde \vp_\go'\|^r\Big)^{\frac{\kp_{r,N}}{r+\kp_{r,N}}} \Big(\|\exp(S_\om(F_N))\|\|\tilde \vp_\go'\|^r\Big)^{\frac{\gk_r}{r+\gk_r}-\frac{\kp_{r,N}}{r+\kp_{r,N}}}\nonumber\\
	&\quad\leq (CK^r)^{\frac{\kp_{r,N}}{r+\kp_{r,N}}} (CK^r)^{\frac{\gk_r}{r+\gk_r}-\frac{\kp_{r,N}}{r+\kp_{r,N}}}=(CK^r)^{\frac{\gk_r}{r+\gk_r}}.	 \label{eq3451}
\end{align}

We now prove the following claim.
\begin{claim} \label{prop32} Let $0<r<+\infty$ be fixed, and let $\gk_r$ be as in Lemma~\ref{lemma101}. Then, we have
	\begin{align*}
		\limsup_{n\to\infty} n V_{n,r}^{\gk_r/r}(m_F) <+\infty.
	\end{align*}
\end{claim}
\begin{proof}
Write $\gh=\frac{\gk_r}{r+\gk_r}$ and $L=(CK^r)^{\frac{\gk_r}{r+\gk_r}}$. As $0<r<+\infty$ is fixed, the number $\gh$ is fixed. Consider an arbitrary positive integer $n\geq 2$, and write
\begin{align}\label{j def Gm_n}
	\Gm_n=\set{\go \in I_N^* :  (m_N(J_{N,\om}) \|\tilde \vp_\go'\|^r)^\gh \leq \frac{L}{n} (\rho_N)^{-1}
	 \;\te{ and }\; (m_N(J_{N,\go^-})\|\tilde \vp_{\go^{-}}'\|^r)^\gh \geq \frac{L}{n} (\rho_N)^{-1}},
\end{align}
where
\begin{align*}
	\rho_N:=(C^{-3}K^{-r} m_N(J_{N,i})\|\tilde \vp_N'\|^r)^\gh.
\end{align*}
Then, $\Gm_n$ is a finite maximal antichain.
Applying Lemma~\ref{lemma921} for $\go \in I_{N}^\ast$, we have
\begin{align*}
	m_N (J_{N,\om}) &= m_N (\tilde \vp_\go(J_N))=\int \exp(S_{\go}(F_N))\,dm_N\geq C^{-1}\|\exp(S_{\go}(F_N))\|\\
	&\geq C^{-3}  \|\exp(S_{\go^-}(F_N))\|\|\exp(S_{\go_{|\go|}}(F_N))\|.
\end{align*}
Again,
\begin{align}\label{j eq 4}
	m_N (J_{N,\om}) =\int \exp(S_{\go}(F_N))\,dm_N\leq \|\exp(S_{\go}(F_N))\|.
\end{align}
Thus, for any $\go \in I_{N}^\ast$, we have
\begin{align}\label{j eq 6}
	m_N (J_{N,\om})\geq C^{-3}m_N (J_{N,\go^-}) m_N (J_{N,\go_{|\go|}}).
\end{align}
In addition, by the Distortion Property we have
\begin{align}\label{j eq 7}
	\|\tilde \vp_\go'\|\geq K^{-1}\|\tilde \vp_{\go^-}'\|\|\tilde \vp_{\go_{|\go|}}'\|.	
\end{align}
Hence, for any $\go \in \Gm_n$, it follows from \eqref{j def Gm_n}, \eqref{j eq 6}, and \eqref{j eq 7} that
\begin{align}\label{j eq 5}
	(m_N(J_{N,\om}) \|\tilde \vp_\go'\|^r)^\gh\geq ( m_N(J_{N,\go^-}) \|\tilde \vp_{\go^-}'\|^r)^\gh \rho_N\geq \frac Ln.
\end{align}
Collecting together \eqref{eq3451}, \eqref{j eq 4}, and \eqref{j eq 5}, we have
\begin{align}\label{j eq 3}
	L\geq \sum_{\go \in \Gm_n}  ( m_N(J_{N,\om}) \|\tilde \vp_\go'\|^r)^\gh\geq \frac Ln \Card(\Gm_n).
\end{align}
Thus, $\Card(\Gm_n) \leq n$. Let $B$ be a set of cardinality $n$, which has points in each of the sets $\tilde \vp_\omega(X)$ for $\omega \in \Gm_n$; this is possible since, we have seen $\Card(\Gm_n) \le n$. Then, in light of Lemma \ref{lemma5566}, \eqref{j def Gm_n}, and \eqref{j eq 3}, for some constant $A>0$, we have
\begin{align*}
	V_{n, r}(m_F) & \leq \int d(x, B)^r\, dm_F\leq \sum_{\go \in \Gm_n} \int \exp(S_{\go}(F_N)) d(\tilde \vp_{\go} (x), B)^r \,dm_F\\
	&\leq A\tilde K \sum_{\go \in \Gm_n} \|\exp(S_{\go}(F_N))\|\|\tilde \vp_{\go}'\|^r\leq A C \tilde K \sum_{\go \in \Gm_n}m_N(J_{N,\go}) \|\tilde \vp_{\go}'\|^r\\
	& = A C \tilde K \sum_{\go \in \Gm_n} \Big( m_N(J_{N,\go}) \|\tilde \vp_{\go}'\|^r\Big)^{1-\gh}\Big(m_N(J_{N,\go}) \|\tilde \vp_{\go}'\|^r\Big)^\gh\\
	& = A C \tilde K \sum_{\go \in \Gm_n} \Big(\left(m_N(J_{N,\go}) \|\tilde \vp_{\go}'\|^r\right)^\eta\Big)^{\frac{1-\gh}{\eta}}\Big(m_N(J_{N,\go}) \|\tilde \vp_{\go}'\|^r\Big)^\gh\\
	&\leq A C \tilde K\Big(\frac{L}{n} (\rho_N)^{-1}\Big)^{\frac{1-\gh}{\gh}} \sum_{\go \in \Gm_n} \Big(m_N(J_{N,\go}) \|\tilde \vp_{\go}'\|^r\Big)^\gh\\
	&\leq A C \tilde KL\Big(\frac{L}{n} (\rho_N)^{-1}\Big)^{\frac{1-\gh}{\gh}}.
\end{align*}
Noting that
\begin{align*}
	\frac{1-\eta}{\eta}\cdot\frac{\kp_r}{r}=1,
\end{align*}
we have
\[
	n V_{n, r}^{\frac{\gk_r}{r}} (m_F) \leq (A C \tilde KL)^{\frac{\gk_r}{r}}\frac{L}{\rho_N}.
\]
Now, recalling the fact that $N$ depends on $r$, and $r$ is fixed, we have
\[
	\limsup_{n\to\infty} n V_{n, r}^{\frac{\gk_r}{r}} (m_F) \leq  (A C \tilde KL)^{\frac{\gk_r}{r}}\frac{L}{\rho_N}<+\infty,
\]
which finishes the claim.
\end{proof}

We now finish the proof of Theorem~\ref{theorem}.

\noindent Recall that
\begin{align*}
	e_{n,r}(\mu):=V_{n,r}^{1/r}(\mu)
\end{align*}
for any Borel probability measure $\mu$  on $\R^d$.
By Proposition 11.3 of \cite{GL1}, we know:
\begin{enumerate}[(a)]
	\item\label{j item a} If $0\leq t<\ul D_r<s$, then
	\begin{align*}
		\lim_{n\to \infty} ne^t_{n, r}=+\infty \text{ and } \liminf_{n\to \infty} ne^s_{n, r}=0.
	\end{align*}
	\item\label{j item b} If $0\leq t<\ol D_r<s$, then
	\begin{align*}
		\limsup_{n\to \infty} ne^t_{n, r}=+\infty \text{ and } \lim_{n\to \infty} ne^s_{n, r}=0.
	\end{align*}
\end{enumerate}
Claim~\ref{prop32} tells us that
\begin{align*}
	\limsup_{n\to\infty} n V_{n, r}^{\frac{\gk_r}{r}} (m_F)<+\infty,
\end{align*}
which, by \eqref{j item b} above, implies  $\ol D_r(m_F)\leq \gk_r$. By Proposition~\ref{prop33}, we have $\ul D_r(m_F)\geq \gk_r$. Hence,
\begin{align*}
	\gk_r\leq  \ul D_r(m_F)\leq \ol D_r(m_F) \leq \gk_r,
\end{align*}
i.e. the quantization dimension $D_r(m_F)$ of the infinite $F$-conformal measure $m_F$ exists and equals $\gk_r$. Note that if $q_r=\frac{\gk_r}{r+\gk_r}$, then by Lemma~\ref{lemma101}, $\gb(q_r)=rq_r$. Thus, it follows that
\begin{align*}
	D_r(m_F)=\frac{\gb(q_r)}{1-q_r}.	
\end{align*}
Hence, the proof of Theorem \ref{theorem} is complete.
\end{proof}

\begin{remark}
Note that our result shows that the $\gk_r$-dimensional upper quantization coefficient of the infinite $F$-conformal measure $m_F$ is finite, but whether the $\gk_r$-dimensional lower quantization coefficient of $m_F$ is positive still remains open.
\end{remark}

We would like to close this section with a class of examples of
summable H\"older families of functions which fulfill our
assumptions.
\begin{example}
	Assume that an infinite conformal iterated function
	system $\Phi=\{\vp_i\}_{i\in I}$ is co-finitely (or hereditarily)
	regular. Let $g:X\to\D R$ be an arbitrary H\"older continuous
	function. Fix $s>\theta_\Phi$. For every $i\in I$ define
	$$
	f^{(i)}(x)=g(x)+s\log|\vp_i'(x)|.
	$$
	Then, $F:=\{f^{(i)}\}_{i\in I}$ is a summable H\"older family of
	functions for which \eqref{eqmu1} holds, and in consequence,
	Theorem~\ref{theorem} is true.
\end{example}

\end{document}